\documentclass{article}

\usepackage[english]{babel}

\usepackage[letterpaper,top=2cm,bottom=2cm,left=3cm,right=3cm,marginparwidth=1.75cm]{geometry}

\usepackage{amsmath,amssymb,amsthm}
\usepackage{graphicx}
\usepackage[colorlinks=true, allcolors=blue]{hyperref}

\usepackage{algorithm} 
\usepackage{algorithmic}  
\usepackage[algo2e]{algorithm2e} 
\usepackage{color}

\newtheorem{thm}{Theorem}[section]
\newtheorem{theorem}{Theorem}[section]
\newtheorem{definition}[thm]{Definition}

\newtheorem{lemma}[thm]{Lemma}

\def\cC{{\cal C}}

\def\cI{{\cal I}}

\def\cM{{\cal M}}

\def\cQ{{\cal Q}}

\def\R{{\mathbb{R}}}

\hypersetup{
	pdftitle={
		Implicit Primal-Dual Guarantees
		in Unconstrained First-Order Minimization
	},
	pdfauthor={
		Benjamin Grimmer and Alex L. Wang
	},
	pdfsubject={
		Stronger primal-dual convergence guarantees and computable dual
		certificates for first-order methods in nonsmooth Lipschitz and
		smooth convex minimization
	},
	pdfkeywords={
		convex optimization,
		unconstrained minimization,
		first-order methods,
		primal-dual guarantees,
		primal-dual gaps,
		dual certificates,
		affine minorants,
		minimizer estimates,
		fixed-step first-order methods,
		subgradient methods,
		gradient methods,
		momentum methods,
		Lipschitz convex minimization,
		smooth convex minimization,
		performance estimation problems,
		worst-case analysis
	},
	pdfcreator={LaTeX with hyperref},
	pdfdisplaydoctitle=true,
	unicode=true
}

\title{Implicit Primal-Dual Guarantees\\ in Unconstrained First-Order Minimization}
\author{
    Benjamin Grimmer\footnote{Johns Hopkins University, Department of Applied Mathematics and Statistics, \texttt{grimmer@jhu.edu}}
    \and
    Alex L.\ Wang\footnote{Purdue University, Daniels School of Business, \texttt{wang5984@purdue.edu}}
}
\date{}

\begin{document}
\maketitle

\begin{abstract}
     This work considers the design of first-order convex optimization algorithms and convergence proofs. In particular, we consider nonsmooth Lipschitz and smooth problems accessed through a subgradient or gradient oracle, respectively. For the general class of fixed-step first-order methods, prior work on Performance Estimation Problems (PEPs) has shown that structured, tight convergence proofs typically exist. Under mild conditions, we further show that any first-order method guaranteeing a bound on the primal objective gap $f(x_N)-f(x_\star)$ assuming only a bound on $\|x_0-x_\star\|$ actually has a stronger guarantee on an explicit, computable primal-dual gap at the same rate. These implicit optimal dual certificates, which take the form of affine lower bounds, also provide insight into the role of auxiliary sequences in momentum methods.
\end{abstract}

\section{Introduction}
We consider the setting of unconstrained convex optimization via first-order methods. Namely, we consider primal problems of the form
\begin{equation} \label{eq:main-problem}
    f_\star = \min_{x\in\R^d} f(x)
\end{equation}
for a given convex function $f\colon \R^d\rightarrow \R$, attained at some minimizer, denoted $x_\star$. A problem instance consists of a function $f$, minimizer $x_\star$, and initialization $x_0$ satisfying $\|x_0-x_\star\|\leq D$ for some $D>0$.
The two problem classes we cover are $M$-Lipschitz nonsmooth convex optimization and $L$-smooth convex optimization.

We will consider iterative methods for solving \eqref{eq:main-problem} given access to either a subgradient oracle or a gradient oracle, depending on the structure assumed on $f$.
Assume that the iterative method produces iterates $x_0,\dots,x_N$ and let $f_i = f(x_i)$ denote the function values and $g_i \in\partial f(x_i)$ denote the subgradients (or gradients) of $f$ at $x_i$.

Our primary goal is to understand the nature of convergence guarantees that iterative methods can provide. 
A substantial body of work in the theoretical optimization literature has developed first-order methods in these settings guaranteeing
\begin{equation}\label{eq:basic-objective-bound}
    f_N - f_\star \leq r.
\end{equation}
Here $r$ represents a convergence rate, typically decreasing to zero as $N$ grows.
In this work, we show that objective guarantees of the form~\eqref{eq:basic-objective-bound} leave substantial structure on the table and can be strengthened ``for free'', i.e., without loss in the numerical rate. 
We say our strengthened bounds are \emph{implicit} since they are implied without additional effort from existing work proving~\eqref{eq:basic-objective-bound}.

Specifically, consider an \emph{arbitrary} first-order method and suppose that~\eqref{eq:basic-objective-bound} holds.
Our strengthened bounds relate the primal iterate $x_N$ with a 
``minimizer estimate'' $z_{N+1}\in\R^d$ and a dual certificate\footnote{This is a dual certificate for the given problem instance, \emph{not} a general PEP certificate.} (formally introduced in Section~\ref{sec:primal-dual}). 
The dual certificate, $\alpha$, is an affine function constructed only from \emph{observed first-order information} with the property that $\alpha(x_\star) \leq f(x_\star)$.
The following two strengthened forms of the guarantee~\eqref{eq:basic-objective-bound} hold: for some $\sigma\geq 0$,
\begin{gather*}
    \underbrace{f_N - f_\star}_{\text{Primal Objective Gap}} \leq \underbrace{f_N - \alpha(x_\star)}_{\text{Primal-Dual Gap}} + \sigma\underbrace{\| z_{N+1} - x_\star \|^2}_{\text{Minimizer Estimate}}\leq r
\end{gather*}
and
\begin{gather*}
    \underbrace{f_N - f_\star}_{\text{Primal Objective Gap}} \leq \underbrace{f_N - \min_{\|x-x_0\|\leq D}\alpha(x)}_{\text{Computable Primal-Dual Gap}} \leq r.
\end{gather*}

The first strengthened bound includes an additional minimizer estimate $z_{N+1}\in\R^d$. As a result, on any instance where \eqref{eq:basic-objective-bound} is tight and $\sigma>0$, one must have $z_{N+1}=x_\star$. 
Conversely, if $\sigma > 0$ and $x_\star\neq z_{N+1}$, the algorithm must strictly outperform the bound \eqref{eq:basic-objective-bound} by an amount growing quadratically in their distance.
The second strengthened bound shows that the primal objective gap, $f_N-f_\star$, can be replaced by a computable primal-dual gap without loss in the numerical rate.

The constructions for $\alpha, z_{N+1},\sigma$ can be made explicit when the algorithm in question is a 
\emph{fixed-step first-order method} (FSFOM) under mild nondegeneracy conditions.
FSFOMs are parameterized by predetermined lower triangular weight matrices $W\in\R^{N\times N}$ and generate their iterates according to the rule:
\begin{equation}
    x_n = x_0 - \sum_{i=0}^{n-1}W_{n-1,i}g_i. \label{eq:general-form-subgrad-methods}
\end{equation}
The FSFOM setup models many classical methods for smooth and nonsmooth convex optimization. It includes simple one-step algorithms like gradient descent and the subgradient method as well as momentum-type algorithms like Nesterov's Fast Gradient Method~\cite{Nesterov1983}, the Optimized Gradient Method (OGM)~\cite{kim2016OGM} and the subgradient SSEP method~\cite{drori2019efficient}. In the last decade, FSFOMs have received particular focus as the Performance Estimation Problem (PEP) framework ensures that they typically have structured convergence proof certificates, providing a convergence guarantee $r$ as in~\eqref{eq:basic-objective-bound}. Such certificates can be numerically computed by semidefinite programming.

The explicit constructions of $\alpha, z_{N+1},\sigma$ allow us to derive equivalent primal-dual understandings of many classic methods. These follow directly from examining existing PEP proofs, strengthening their classic guarantee. As one particularly surprising insight, the minimizer estimate $z_{N+1}$ is exactly the auxiliary point generated by momentum methods. For OGM as an example, our theory identifies that the method can be equivalently understood as tracking a primal solution $x_n$ and dual affine function $\alpha_n$, updated by alternating averaging with descent
\begin{align}
	x_n &= \frac{\tau_{n-1}}{\tau_n} \left(x_{n-1} - \frac{1}{L} g_{n-1} - \frac{\tau_n - \tau_{n-1}}{L}\nabla \alpha_n\right) +  \frac{\tau_n - \tau_{n-1}}{\tau_n}x_0 \nonumber \\
	\alpha_{n+1} &= \frac{\tau_{n-1}}{\tau_n} \alpha_n +  \frac{\tau_n - \tau_{n-1}}{\tau_n}\left(f(x_n) + \langle g_n, \cdot-x_n\rangle + \frac{1}{2L}\|g_n\|^2\right) \label{eq:affine-aggregating-step-OGM},
\end{align}
where $L$ is the smoothness constant of $f$ and $\tau_n$ is a parameter sequence defined in~\eqref{eq:OGM_tau_def}. In this form, OGM's guarantee can be strengthened, ensuring
\begin{equation*}
	\underbrace{f_N - f_\star}_{\text{Primal Objective Gap}} \leq \underbrace{f(x_N) - \min_{\|x-x_0\|\leq D} \alpha_{N+1}(x)}_{\text{Computable Primal-Dual Gap}} \leq \frac{LD^2}{2\tau_N} \sim \frac{LD^2}{N^2},
\end{equation*}
which is exactly the minimax optimal rate for primal objective gap convergence.

{\bf Outline.} Sections~\ref{sec:pep} and~\ref{sec:primal-dual} introduce the two key tools our development utilizes, the structured convergence analysis framework of performance estimation and an elementary primal-dual formulation of unconstrained minimization problems~\eqref{eq:main-problem}, respectively. Section~\ref{sec:general-results} shows that improved guarantees of the above forms are necessary for any general first-order method in smooth or nonsmooth optimization. 
Sections~\ref{sec:subgrad-development} and~\ref{sec:grad-development} then specialize to the setting of fixed-step methods, where we can characterize $\alpha, z_{N+1}, \sigma$ more explicitly, for Lipschitz convex minimization and smooth convex minimization, respectively.

\subsection{Related Work}
This work is in line with the recent movement in the first-order optimization literature to use duality in algorithm design and analysis. A structured family of methods including the classic subgradient method and fast gradient method was analyzed in a unified fashion via a perturbed Fenchel duality approach by Gutman and Pe\~{n}a~\cite{gutman2022}. A wider approach to algorithmic design from duality was proposed by Diakonikolas and Orecchia~\cite{Diakonikolas2019}. There, they provided a general framework for analyzing continuous-time dynamics for which an associated duality gap converges. From this, associated iterative methods have primal-dual convergence guarantees given by controlling or canceling discretization error terms. This approximate duality gap technique applies beyond the scope of problems considered here, providing algorithms and analyses, for example, covering composite/constrained problems, convex-concave saddle point problems, and variational inequalities.

Particular focus has been placed on the development of a dual understanding of momentum in smooth optimization. Nesterov~\cite{Nesterov1983}'s fast gradient method possesses an estimate sequence analysis that aggregates lower bounding affine models (with an initial quadratic term). Here our lower bounds are similarly built by aggregating affine models, e.g.,~\eqref{eq:affine-aggregating-step-OGM}. An accelerated algorithm based on optimally averaging quadratic lower bounds (for smooth strongly convex problems) was proposed by Drusvyatskiy, Fazel, and Roy~\cite{Drusvyatskiy2018}. Allen-Zhu and Orecchia~\cite{AllenZhu2017} analyzed Nesterov's acceleration through a linear coupling approach: making primal progress through gradient descent steps and dual progress through mirror descent steps. 

More recent works have taken primal-dual and game-theoretic approaches to understand momentum. Lan and Zhou~\cite{LanZhou2018} reformulate their main problem into a primal-dual saddle point problem and then describe momentum as extrapolated steps on the primal side. From this, they proposed new accelerated schemes using gradient extrapolations on the dual side. Wang, Abernethy, and Levy~\cite{WangAbernethyLevy2024} directly reformulate convex optimization as a Fenchel game where no-regret strategies correspond to classical methods. Guarantees for these methods can then be extracted from associated regret bounds. Most recently, Burns and Liang~\cite{burns2026} designed a new accelerated method able to provide computable certificates of accuracy. There, they interpret the resulting primal-dual averaging schemes through zero-sum games and Fisher markets.

These prior works either provided primal-dual approaches to design new algorithms or new understandings of classical methods. In contrast, we provide a general mechanism from which dual phenomena must arise for first-order methods. Stronger primal-dual convergence guarantees necessarily exist for arbitrary algorithms with objective gap guarantees. For any fixed-step method with an appropriate PEP proof, our theory further gives these necessary primal-dual strengthened bounds in closed form.

\section{A Review of Performance Estimation Problems} \label{sec:pep}

The performance estimation framework, pioneered by Drori and Teboulle~\cite{drori2014performance} and Taylor, Hendrickx, and Glineur~\cite{taylor2017CompositePEP,taylor2017InterpolationFMuLPEP}, provides a means to construct a worst-case instance $(f,x_0,x_\star)$ for a fixed-step method (defined by stepsize weights $W$) against a given structured problem class. The key observation is that when such a method is run, the only relevant aspects of the problem instance are the function values and (sub)gradients at the iterates $x_i$. We denote these by
\begin{equation*}
    x_i,\qquad f_i=f(x_i), \qquad g_i\in\partial f(x_i).
\end{equation*}
Further, denote the first-order information that the problem instance provides at the given minimizer by $x_\star$, $f_\star$, and $g_\star=0$.
Let $\cI=\{\star,0,\dots,N\}$.

Using this first-order information, a few useful nonnegative quantities can be defined.
The distance bound implies
\begin{equation}
    \mathcal{D} := D^2 - \|x_0-x_\star\|^2 \geq 0.
\end{equation}
Convexity of the function $f$, and in particular the subgradient inequality, guarantees that the following quantity is nonnegative for every pair of indexed points
\begin{equation}
    \mathcal{C}_{i,j} := f_i - f_j - \langle g_j, x_i-x_j\rangle \geq 0\qquad \forall i,j\in\cI.
\end{equation}
In the case that $f$ is $M$-Lipschitz, one can additionally guarantee that
\begin{equation}
    \mathcal{M}_{i} = M^2 - \|g_i\|^2 \geq 0 \qquad \forall i\in\cI.
\end{equation}
For $L$-smooth convex functions with $L>0$ (defined as $\nabla f$ being $L$-Lipschitz), the convexity inequality can be strengthened to the following cocoercivity-type inequality
\begin{equation}\label{eq:cocoercive}
    \mathcal{Q}_{i,j} := f_i - f_j - \langle g_j, x_i-x_j\rangle -\frac{1}{2L}\|g_i-g_j\|^2 \geq 0\qquad \forall i,j\in\cI.
\end{equation}

The fundamental insight behind the PEP approach to analyzing FSFOMs is the fact that the nonnegativity of these quantities suffices for any analysis based only on observed values. The following interpolation theorems formalize this.
\begin{theorem}[Theorem 3.5 of~\cite{taylor2017CompositePEP}]\label{thm:LipschitzConvexInterpolation}
    The first-order observations $(x_i,f_i,g_i)_{i\in\cI}$ 
    can be interpolated by an $M$-Lipschitz convex function if and only if $\mathcal{C}_{i,j}\geq 0$ for all $i,j\in\cI$ and $\mathcal{M}_i\geq 0$ for all $i\in\cI$.
\end{theorem}
\begin{theorem}[Theorem 4 of~\cite{taylor2017InterpolationFMuLPEP}]\label{thm:SmoothConvexInterpolation}
    The first-order observations $(x_i,f_i,g_i)_{i\in\cI}$ can be interpolated by an $L$-smooth convex function if and only if $\mathcal{Q}_{i,j}\geq 0$ for all $i,j\in\cI$.
\end{theorem}

Using these interpolation theorems, Taylor, Hendrickx, and Glineur showed that one can equivalently cast the problem of computing a worst-case instance as computing worst-case first-order oracle values satisfying these conditions. For example, for a FSFOM with weights $W$, the worst-case final objective gap over $M$-Lipschitz convex problems is equal to
\begin{align}
    \mathtt{PEP}_{M}(W) &:= \begin{cases}
        \max_{f,(x_i,g_i)_{i\in\cI}} &f(x_N)-f(x_\star)\\
        \mathrm{s.t.} & x_n = x_0 - \sum_{i=0}^{n-1}W_{n-1,i}g_i \quad \forall n=1,\dots,N\\
        & g_i\in\partial f(x_i) \qquad \forall i\in\cI\\
        & f\text{ is convex, $M$-Lipschitz, and minimized at $x_\star$}\\
        & \|x_0-x_\star\|\leq D
    \end{cases} \nonumber\\
    &= \begin{cases}
        \max_{(x_i,f_i,g_i)_{i\in\cI}} & f_N - f_\star\\
        \mathrm{s.t.} & x_n = x_0 - \sum_{i=0}^{n-1}W_{n-1,i}g_i \quad \forall n=1,\dots,N\\
        & g_\star = 0\\
        & \mathcal{C}_{i,j} \geq 0 \qquad \forall i,j\in\cI\\
        & \mathcal{M}_i \geq 0 \qquad \forall i\in\cI\\
        & \mathcal{D} \geq 0.
    \end{cases} \label{eq:subgradient-pep}
\end{align}
This problem can be written as a semidefinite program. To see this, consider substituting the defining equalities for $g_\star$ and $x_n$, $n=1,\dots,N$, in the remaining constraints. The resulting problem is linear in the function values $f=(f_\star,f_0,\dots,f_N)$ and linear in the quadratic form
\begin{equation}
    G = P^\top P, \qquad P = \begin{bmatrix}
        x_0-x_\star, g_0,\dots, g_N 
    \end{bmatrix}
\end{equation}
which must be positive semidefinite by definition. Provided $d\geq N+2$, any positive semidefinite $G$ can be factored to recover the underlying $P$. Hence, for $d$ large enough, $\mathtt{PEP}_{M}(W)$ is a semidefinite program. For $d<N+2$, the resulting semidefinite program still provides a valid upper bound.

The dual of this semidefinite program then optimizes over upper bounds $r$ on the worst-case value of $f_N - f_\star$. This dual selects nonnegative multipliers $\lambda_{i,j}$, $\gamma_{i}$, and $\sigma$ for the inequality constraints and a positive semidefinite multiplier $Z$ for the $G \succeq 0$ constraint.
These multipliers are dual feasible if they satisfy the polynomial identity
\begin{equation*}
    r - (f_N - f_\star) = \sum_{i,j\in\cI} \lambda_{i,j} \mathcal{C}_{i,j} + \sum_{i\in\cI} \gamma_i\mathcal{M}_i + \sigma\mathcal{D} + \langle Z, G\rangle,
\end{equation*}
fixing $x_n = x_0 - \sum_{i=0}^{n-1}W_{n-1,i}g_i$ and $g_\star=0$. Here in this identity, we treat the values $f_i$ and the entries of $G$ as free variables. In these terms, the identity is affine. Note this identity combined with the nonnegativity of the right-hand-side proves the guarantee $r$.
If strong duality (with dual attainment) holds, then $r=\mathtt{PEP}_{M}(W)$ and 
an optimal solution $(\lambda_{i,j},\gamma_i,\sigma,Z)$ provides a structured proof of this convergence rate.

Analogous primal and dual semidefinite programs can be constructed for $L$-smooth convex functions: simply replace the constraints $\cC_{i,j}\geq 0$ and $\cM_i\geq 0$ with the constraints $\cQ_{i,j}\geq 0$.

\section{The Primal-Dual View of Unconstrained Minimization}  \label{sec:primal-dual}
Here, we develop convergence theory in terms of stronger primal-dual guarantees than the primal-only objective guarantees natural in the PEP framework. Our development follows the primal-dual minimax perspective taken by~\cite{GrimmerLi2025} in the analysis of certain subgradient methods. At first glance, an unconstrained convex minimization problem of the form~\eqref{eq:main-problem} does not possess a natural primal-dual perspective. However, for full domain convex functions (which must be closed convex and proper), one can describe $f$ as the supremum over its affine minorants~\cite[Theorem 12.1]{Rockafellar1970}. Letting $A_f$ denote the set of all affine minorants
\begin{equation*}
    A_f = \{\alpha \colon \R^d\rightarrow \R\mid \alpha(x)=\langle a,x\rangle+b,\ \alpha \leq f\},
\end{equation*}
one has that $f(x) = \sup_{\alpha\in A_f}\alpha(x)$. Hence, our basic minimization problem can be written as
\begin{equation*}
    f_\star = \min_x f(x) = \min_x\sup_{\alpha\in A_f} \alpha(x).
\end{equation*}
Given any $\alpha\in A_f$, a lower bound on $f_\star$ is immediately provided by $\alpha(x_\star)$. A lower bound independent of $x_\star$ can be produced by considering the dual maximin problem of
\begin{equation*}
    f_\star = \min_x\sup_{\alpha\in A_f} \alpha(x) \geq \sup_{\alpha\in A_f} \min_x \alpha(x) = \sup_{\alpha\in A_f} \begin{cases}
        \alpha(x_0) & \text{if }\nabla \alpha =0\\
        -\infty & \text{otherwise}
    \end{cases} =: d_\star
\end{equation*}
where the second equality observes that the minimum of any nonconstant affine function is $-\infty$. Note that since $\alpha$ is affine, we omit the evaluation point from the gradient $\nabla \alpha$ in our notation.

From the above calculation, the dual problem seeks the largest constant function (i.e., $\nabla \alpha=0$) underneath $f$. The dual maximum is attained by $\alpha_\star(x) = \langle 0,x\rangle+f_\star$. So strong duality holds, with $f_\star=d_\star$. The above dual is too simple to provide a meaningful source of lower bounding certificates, however, as the bound provided by a dual certificate $\alpha$ is $-\infty$ for all but constant functions.

This primal-dual lens can be enriched in two ways.
First, when a bound $D$ on the primal distance to optimal is known, one can add the primal constraint $\|x - x_0\|\leq D$. 
Second, we may relax the dual constraint $\alpha\in A_f$ to only requiring that $\alpha(x_\star)\leq f(x_\star)$ (i.e., $\alpha$ lower bounds $f$ at its minimizer).
The former will be useful in addressing $M$-Lipschitz settings, while the latter will find use in $L$-smooth settings. The following two subsections develop these models.

\subsection{Stronger Certificates Given a Primal Distance Bound $D$}
Given some $D>0$ such that the problems of interest have a minimizer $x_\star$ with bounded distance $\|x_0-x_\star\|\leq D$, we can strengthen our dual. Without loss, we can add the primal minimization constraint that $x$ lies in the ball $\|x - x_0\|\leq D$. Hence, we can write a primal-dual problem pair for a given problem instance with bounded distance as
\begin{equation}
    f_\star = \min_{\|x-x_0\|\leq D} \sup_{\alpha\in A_f} \alpha(x) \geq \sup_{\alpha\in A_f} \min_{\|x-x_0\|\leq D} \alpha(x) =: d_\star. \label{eq:main-primal-dual-pair}
\end{equation}
Since $\alpha$ is affine, we can solve the inner minimization over $x$ in closed form. It is attained at $x=x_0-D \nabla \alpha/\|\nabla \alpha\|$ for any given nonconstant $\alpha$, making the dual problem $\sup_{\alpha\in A_f} \alpha(x_0) - D\|\nabla \alpha\|$.
Again, strong duality holds here, attained by the constant function $\alpha_\star(x) = \langle 0,x\rangle + f_\star$.

From~\eqref{eq:main-primal-dual-pair}, we have a structured primal-dual pair of outputs that algorithms can provide to certify solution quality. Given a dual solution $\alpha\in A_f$, one can lower bound $f_\star$ by the following two bounds
\begin{equation}\label{eq:dual-bounding-sequence}
    f_\star \geq \alpha(x_\star) \geq \min_{\|x-x_0\|\leq D} \alpha(x)
\end{equation}
where the first is tighter but the second is independent of $x_\star$. Hence, given a primal solution $x$ and a dual solution $\alpha$, the primal objective gap at $x$ is bounded by the primal-dual gaps of
\begin{equation}\label{eq:primal-dual-gap}
    f(x) - f(x_\star) \leq \underbrace{f(x) - \alpha(x_\star)}_{\text{Primal-Dual Gap}} \leq \underbrace{f(x) - \min_{\|x-x_0\|\leq D} \alpha(x)}_{\text{Computable Primal-Dual Gap}}.
\end{equation}
Provided a bound $D$ is known, this final quantity is computable and independent of $x_\star$.

Constructing elements of $A_f$ is straightforward for first-order methods. Each first-order oracle query directly provides an affine minorant by the definition of a $g_i\in\partial f(x_i)$. Namely, one must have
$$ f(x) \geq f(x_i) + \langle g_i, x-x_i\rangle \qquad \forall x\in\R^d.$$
So at every iteration, an affine minorant $\alpha(x) = f(x_i) + \langle g_i, x-x_i\rangle$ is observed. Moreover, noting that the set $A_f$ is convex, any convex combination of subgradient lower bounds can provide a candidate dual solution $\alpha$.

\subsection{Certificates via ``Minorants at the Minimizer'' Given Smoothness}
Smoothness enables an even stronger class of dual models to be constructed for our unconstrained minimization. For $L$-smooth convex functions, each gradient provides a stronger lower bound at the minimizer than the simple subgradient inequality. Namely, the cocoercive-type inequality (i.e., $\mathcal{Q}_{\star,i}\geq 0$ in~\eqref{eq:cocoercive}) ensures that every gradient $g_i = \nabla f(x_i)$ has
$$ f(x_\star) \geq f(x_i) + \langle g_i, x_\star-x_i\rangle + \frac{1}{2L}\|g_i\|^2.$$
Note this inequality only holds at $x_\star$, not all $x$. So at every iteration, gradient methods obtain an affine ``minorant at the minimizer'' given by 
$\alpha(x) = f(x_i) + \frac{1}{2L}\|g_i\|^2 + \langle g_i, x-x_i\rangle$.

(This extra term in the gradient norm squared was recently observed by~\cite{FloreaOptimalLowerBound} to be key to strengthening momentum methods to attain the exactly optimal convergence rate of OGM.)

For our primal-dual reasoning above, all that is needed for $\alpha$ is that it is no greater than $f$ at $x_\star$ for~\eqref{eq:dual-bounding-sequence} to hold. Hence, for $L$-smooth convex functions, one can consider any affine ``minorant at the minimizer'', denoted
\begin{equation*}
    A_f^\star = \{\alpha \colon \R^d\rightarrow \R\mid \alpha(x)=\langle a,x\rangle+b,\ \alpha(x_\star) \leq f(x_\star)\}.
\end{equation*}
We refer to such $\alpha\in A_f^\star$ as an affine ``minorant at the minimizer'', given that it is only guaranteed to lie below $f$ at $x_\star$.

Note then that $\sup_{\alpha\in A_f^\star} \alpha(x)$ is $+\infty$ for all inputs other than $x_\star$, where it takes value $f_\star$. Although this supremum is typically different from $f$, minimizing it still returns $f_\star$. So we have strong duality between the primal-dual pair $\min_{\|x-x_0\|\leq D} \sup_{\alpha\in A_{f}^\star} \alpha(x) = \sup_{\alpha\in A_{f}^\star} \min_{\|x-x_0\|\leq D} \alpha(x)$.

Again, observing that $A_f^\star$ is convex, one can construct dual solutions from any first-order method by aggregating the affine ``minorants at the minimizer'' generated by observed gradients at runtime $\alpha(x) = f(x_i) + \frac{1}{2L}\|g_i\|^2 + \langle g_i, x-x_i\rangle$.

\section{Existence of Dual Certificates and Minimizer Estimates} \label{sec:general-results}

In this section, we consider any first-order method (not necessarily fixed-step, deterministic, or otherwise structured). Suppose the method has been run on some problem instance, observing first-order information $(x_i,f_i,g_i)_{i=0}^N$ and reporting a primal objective guarantee against all instances consistent with the given first-order observations of
\begin{equation} \label{eq:simple-primal-guarantee}
    f_N - f_\star \leq r.
\end{equation}
Note that just as the method is left general, how one derives this convergence guarantee is left general. It need not come from a structured PEP proof of the form discussed in Section~\ref{sec:pep}. Any procedure resulting in a uniform guarantee against every problem instance consistent with the fixed, observed values $(x_i,f_i,g_i)_{i=0}^N$ is fine.

Here we will show that, in both the nonsmooth and smooth settings, one can always deduce stronger primal-dual guarantees from this observed information. The key object in both settings will be to consider all possible placements of $x_\star$ consistent with the first-order information. Optimizing over $x_\star$'s placement will yield dual insights.

\subsection{Dual Certificates in Nonsmooth Optimization}
First we consider convex (potentially nonsmooth) minimization given a distance bound $D>0$. Here we do not make any Lipschitz continuity assumptions, although these will be required in Section~\ref{sec:subgrad-development} when we consider fixed-step methods. Suppose an algorithm observed first-order information $(x_i,f_i,g_i)_{i=0}^N$. We know that whatever values $(x_\star, f_\star)$ take, they must satisfy $\mathcal{C}_{i,\star} \geq 0$ and $\mathcal{C}_{\star,i} \geq 0$ for each $i=0,\dots,N$. From this, the worst possible final objective gap is given by
\begin{equation} \label{eq:objective-gap-maximizing}
	r_\star = \begin{cases}
		\max_{x_\star, f_\star} & f_N - f_\star\\
		\mathrm{s.t.} & D^2 - \|x_0-x_\star\|^2 \geq 0\\
		& f_\star - f_i -\langle g_i, x_\star-x_i\rangle \geq 0 \qquad \forall i=0,\dots,N\\
		&  f_i - f_\star \geq 0 \qquad \forall i=0,\dots,N.
	\end{cases}
\end{equation}
Note this is not a relaxation, but exact, as given maximizers $x_\star,f_\star$, one can set $g_\star=0$ and apply the Interpolation Theorem~\ref{thm:LipschitzConvexInterpolation} to construct an instance realizing this worst-case gap. So $r_\star\leq r$.

Let $\ell_i(x)$ denote $f_i + \langle g_i, x - x_i\rangle$, the affine minorant of any interpolating $f$ provided by the subgradients $g_i$. Note $\mathcal{C}_{\star,i} \geq 0$ ensures that $f_\star \geq \ell_i(x_\star)$ for each $i=0,\dots,N$. Taking the maximum over indices $i$, we have that $f$ must lie above the cutting-plane model constructed from each of these affine minorants. In particular, for any placement of $x_\star$, we have
$$ f(x_\star) \geq \max_{i=0,\dots,N} \ell_i(x_\star). $$

Below we find that the maximum final objective gap (i.e., smallest possible final value of $f(x_\star)$), consistent with the observed first-order information, is given by minimizing over all placements of $x_\star$ in this cutting plane model. By resolving the associated minimax problem, we arrive at guarantees of our strengthened primal-dual forms.

\begin{theorem}
	Consider any observed first-order information $(x_i,f_i,g_i)_{i=0}^N$ consistent with some convex function possessing a minimizer with $\|x_0-x_\star\|\leq D$. 
	Then there exist nonnegative weights $a_i$, summing to one, such that $\alpha(x) = \sum_{i=0}^N a_i (f_i + \langle g_i, x - x_i\rangle) \in A_f$ has
	\begin{align}
		\underbrace{f_N - f_\star}_{\text{Primal Objective Gap}} &\leq \underbrace{f_N - f_\star}_{\text{Primal Objective Gap}} + \underbrace{\sigma\|z_{N+1} - x_\star\|^2}_{\text{Minimizer Estimate}} \\
		&\leq \underbrace{f_N - \alpha(x_\star)}_{\text{Primal-Dual Gap}} + \underbrace{\sigma\|z_{N+1} - x_\star\|^2}_{\text{Minimizer Estimate}} \\
		&\leq \underbrace{f_N - \min_{\|x-x_0\|\leq D} \alpha(x)}_{\text{Computable Primal-Dual Gap}}\\
		&\leq r_\star
	\end{align}
	where if $\nabla \alpha\neq 0$, $\sigma = \frac{\|\nabla \alpha\|}{2D}$, $z_{N+1} = x_0 - \frac{1}{2\sigma}\nabla \alpha$, otherwise $\sigma=0$ and $z_{N+1}$ can be set arbitrarily.
\end{theorem}
\begin{proof}
	First, notice that since $(x_i,f_i,g_i)_{i=0}^N$ is consistent with some function $\hat{f}$ minimized at $\hat{x}_\star$, the problem~\eqref{eq:objective-gap-maximizing} is feasible---consider setting $f_\star=\hat{f}(\hat{x}_\star)$ and $x_\star=\hat{x}_\star$.
	Moreover, the final constraints in~\eqref{eq:objective-gap-maximizing}, corresponding to $\mathcal{C}_{i,\star} \geq 0$, can be omitted without loss. This follows since any global maximizer with these omitted has $f_\star \leq \hat{f}(\hat{x}_\star)$, so $f_i - f_\star = (f_i - \hat{f}(\hat{x}_\star)) + (\hat{f}(\hat{x}_\star) - f_\star)\geq 0$ and so is feasible for the full problem. 
	
	Thus we can simplify the definition of $r_\star$ as follows
	\begin{align*}
		r_\star
		&= \max_{\|x_\star-x_0\| \leq D} \min_{i=0,\dots,N} f_N - \ell_i(x_\star)\\
		&= \max_{\|x_\star-x_0\| \leq D} \min_{a\in\Delta_{N+1}} f_N - \sum_{i=0}^N a_i \ell_i(x_\star)\\
		&= \min_{a\in\Delta_{N+1}} \max_{\|x_\star-x_0\| \leq D} f_N - \sum_{i=0}^N a_i \ell_i(x_\star)\\
		&= \min_{a\in\Delta_{N+1}} f_N - \alpha_a(x_0) + D\|\nabla \alpha_a\|
	\end{align*}
	where the first equality sets $f_\star$ to its minimum allowable value, the second replaces the finite minimum with minimization over the simplex $\Delta_{N+1}=\{a\in\mathbb{R}^{N+1} \mid \sum_{i=0}^N a_i = 1, a\geq 0\}$, the third line notes that this final problem is a compact, convex-concave minimax problem, so we can exchange the order of minimization and maximization, and the final line computes the minimum of this affine function $\alpha_a = \sum_{i=0}^N a_i\ell_i$ over a ball.
	
	Note that since the simplex is compact, some minimizing $a\in\Delta_{N+1}$ must exist.
	If this optimal $a$ has $\nabla \alpha_a=0$, setting $\sigma=0$, the claimed sequence of inequalities all hold immediately. Now suppose $\nabla \alpha_a\neq 0$. Then the claimed inequalities hold with $\sigma=\frac{\|\nabla \alpha_a\|}{2D}$ and $z_{N+1}=x_0-\frac{1}{2\sigma}\nabla \alpha_a$ as
	\begin{align*}
		r_\star &= f_N - \alpha_a(x_0) + \|\nabla \alpha_a\|D\\
		&= f_N - \alpha_a(x_0) + \sigma D^2 + \frac{1}{4\sigma}\| \nabla \alpha_a \|^2\\
		&= f_N - \alpha_a(x_\star) + \sigma (D^2 - \|x_0-x_\star\|^2) + \sigma \|z_{N+1}-x_\star\|^2\\
		&\geq f_N - \alpha_a(x_\star) + \sigma \|z_{N+1}-x_\star\|^2\\
		&\geq f_N - f_\star + \sigma \|z_{N+1}-x_\star\|^2.\qedhere
	\end{align*}
\end{proof}

\subsection{Dual Certificates in Smooth Convex Optimization}
We now repeat the above development for the case where the observed first-order information comes from an $L$-smooth convex function.

As before, we consider all possible placements of $(x_\star,f_\star)$ consistent with the available first-order information. Recalling the necessary $\mathcal{Q}$ inequalities from~\eqref{eq:cocoercive} (which are sufficient for interpolation), any placement of the minimizer and minimum value $(x_\star,f_\star)$ must satisfy, for each $i=0,\dots,N$,
\begin{align*}
    \mathcal{Q}_{\star,i}
    &=
    f_\star - f_i - \langle g_i, x_\star-x_i\rangle
    -\frac{1}{2L}\|g_i\|^2
    \geq 0,\\
    \mathcal{Q}_{i,\star}
    &=
    f_i - f_\star
    -\frac{1}{2L}\|g_i\|^2
    \geq 0,
\end{align*}
together with the distance bound $\mathcal{D}=D^2-\|x_0-x_\star\|^2\geq 0$. So the largest final objective gap compatible with the observed first-order information is
\begin{equation} \label{eq:smooth-objective-gap-maximizing}
    r_\star = \begin{cases}
        \max_{x_\star,f_\star} & f_N - f_\star\\
        \mathrm{s.t.} &
        f_\star - f_i - \langle g_i, x_\star-x_i\rangle
        -\dfrac{1}{2L}\|g_i\|^2
        \geq 0
        \qquad \forall i=0,\dots,N,\\[0.4em]
        &
        f_i - f_\star
        -\dfrac{1}{2L}\|g_i\|^2
        \geq 0
        \qquad \forall i=0,\dots,N,\\[0.4em]
        &
        D^2-\|x_0-x_\star\|^2 \geq 0.
    \end{cases}
\end{equation}

The direct interpretation of~\eqref{eq:smooth-objective-gap-maximizing} is again that $f_N-f_\star\leq r_\star$, so any reported primal guarantee $f_N-f_\star\leq r$ must have $r\geq r_\star$. 

Comparing \eqref{eq:smooth-objective-gap-maximizing} with \eqref{eq:objective-gap-maximizing}, we see that the only difference is replacing $f_i$ with $f_i + \frac{1}{2L}\|g_i\|^2$ in the first inequalities and replacing $f_i$ with $f_i - \frac{1}{2L}\|g_i\|^2$ in the second inequalities. Again, considering a function $\hat{f}$ consistent with the given information, one can conclude this problem is feasible and that the second collection of inequalities corresponding to $\mathcal{Q}_{i,\star}\geq 0$ can be omitted.

Thus, by taking the same minimax view as before, stronger guarantees can again be extracted. We omit its proof, which follows verbatim after making the above substitutions and considering affine minorants at the minimizer $m_i(x) = f_i + \frac{1}{2L}\|g_i\|^2 + \langle g_i, x-x_i\rangle$ instead of $\ell_i$.

\begin{theorem}
    Consider any observed first-order information $(x_i,f_i,g_i)_{i=0}^N$ consistent with some $L$-smooth convex function possessing a minimizer with $\|x_0-x_\star\|\leq D$. Then 
    there exist nonnegative weights $a_i$, summing to one, such that $\alpha(x) = \sum_{i=0}^N a_i (f_i +\frac{1}{2L}\|g_i\|^2 + \langle g_i, x - x_i\rangle) \in A_f^\star$ has 
    \begin{align}
    \underbrace{f_N - f_\star}_{\text{Primal Objective Gap}} &\leq \underbrace{f_N - f_\star}_{\text{Primal Objective Gap}} + \underbrace{\sigma\|z_{N+1} - x_\star\|^2}_{\text{Minimizer Estimate}} \\
    &\leq \underbrace{f_N - \alpha(x_\star)}_{\text{Primal-Dual Gap}} + \underbrace{\sigma\|z_{N+1} - x_\star\|^2}_{\text{Minimizer Estimate}} \\
    &\leq \underbrace{f_N - \min_{\|x-x_0\|\leq D} \alpha(x)}_{\text{Computable Primal-Dual Gap}}\\
    &\leq r_\star
    \end{align}
    where if $\nabla\alpha\neq 0$, $\sigma = \frac{\|\nabla \alpha\|}{2D}$, $z_{N+1} = x_0 - \frac{1}{2\sigma}\nabla \alpha$, otherwise $\sigma=0$ and $z_{N+1}$ can be set arbitrarily.
\end{theorem}

So in the smooth convex setting, the same dynamic placement argument again yields a minimizer estimate and a dual certificate directly from the observed first-order information. The main distinction from the nonsmooth case is that smoothness allows an extra term in the minorant at the minimizer.

We summarize the results of this section: Consider an arbitrary procedure for producing iterates $x_0,x_1,\dots,x_N$. After the fact, once one has observed the first-order information from running the algorithm, one can deduce a minimizer estimate $z_{N+1}$ and a dual model $\alpha$ by solving an instance-dependent minimax problem, dual to~\eqref{eq:objective-gap-maximizing} or~\eqref{eq:smooth-objective-gap-maximizing}. These candidate minimizers and dual certificates directly establish stronger performance guarantees than~\eqref{eq:simple-primal-guarantee} at no cost, i.e., without harming worst-case guarantees.
In the remainder of this paper, we show that for nondegenerate FSFOMs, an \emph{instance-independent} construction of $\alpha$ and $z_{N+1}$ can be determined \emph{a priori} from the algorithm's PEP.

\section{Dual Constructions for Fixed-Step Subgradient Methods} \label{sec:subgrad-development}
This section develops general conditions under which subgradient-type methods with a primal objective gap guarantee have the same guarantee on a primal-dual gap, often being strictly stronger. The section then concludes with three example subgradient methods from the literature, now with strengthened primal-dual theory.

\subsection{Primal-Dual Guarantees for Lipschitz Convex Minimization}

Consider any fixed-step subgradient method, defined by weights $W$. Recall its worst-case final objective gap is given by $\mathtt{PEP}_{M}(W)$, defined in~\eqref{eq:subgradient-pep}. Further, suppose $d\geq N+2$ and that strong duality holds for its associated semidefinite programming problem. Establishing the strict feasibility of these PEPs suffices to establish strong duality and dual attainment. The lower triangular entries of $W$ being nonzero is sufficient to this end~\cite[Lemma 1.1]{zoll2026complete}. For example, Theorem 6 of~\cite{taylor2017InterpolationFMuLPEP} establishes the analogous result for smooth convex settings.

The dual problem then establishes bounds $\mathtt{PEP}_{M}(W)\leq r$ by providing a corresponding dual feasible solution: nonnegative multipliers $\lambda_{i,j}, \gamma_i, \sigma$ (for $\mathcal{C}_{i,j},\mathcal{M}_i,\mathcal{D}$, respectively) and positive semidefinite $Z\in\mathbb{S}^{N+2}$. Given strong duality and dual attainment, there must exist multipliers satisfying the following identity:
\begin{equation} \label{eq:subgrad-primal-gap-target}
    \mathtt{PEP}_{M}(W) - (f_N-f_\star) = \sum_{i,j\in\cI} \lambda_{i,j} \mathcal{C}_{i,j} + \sum_{i\in\cI} \gamma_i\mathcal{M}_i + \sigma\mathcal{D} + \langle Z, G\rangle,
\end{equation}
where $x_n = x_0 - \sum_{i=0}^{n-1}W_{n-1,i}g_i$ and $g_\star=0$, and the remaining data ($f_i$ and each inner product in $G$) are treated as free variables.
These multipliers establish that the right-hand side is nonnegative for any first-order information $(x_i,f_i,g_i)$ encountered. Hence, the left-hand side is nonnegative, i.e., they prove a tight convergence rate of $f_N-f_\star \leq \mathtt{PEP}_{M}(W)$.

The tightness of this convergence rate is easily observed. Let $x_i,f_i,g_i$ be first-order data attaining the optimal value of the primal PEP~\eqref{eq:subgradient-pep}. By the Interpolation Theorem~\ref{thm:LipschitzConvexInterpolation}, there exists an $M$-Lipschitz convex problem matching these first-order values. Then, running the subgradient method defined by $W$ on this particular interpolated instance will observe first-order values
\begin{equation}
    x_i^W = x_i, \qquad f_i^W = f_i, \qquad g_i^W = g_i.
\end{equation}
So the final objective gap realized by $W$ will be $f_N^W-f_\star^W = \mathtt{PEP}_{M}(W)$.

We next apply complementary slackness to simplify a given primal-dual PEP pair (a tight instance and a structured proof). If the realized primal problem data has $\mathcal{C}_{i,j}^W > 0$, then the corresponding dual multiplier $\lambda_{i,j}$ must be zero.
In particular, suppose that no iterate is prematurely optimal \emph{on all worst-case functions}. Equivalently, suppose that for each $i=0,\dots,N$, there exists some worst-case function $f$ such that $f_i^W > f_\star^W$. Then, since $g_\star=0$, we have that $\mathcal{C}_{i,\star} > 0$ and so $\lambda_{i,\star}=0$. Also note that since $\mathcal{C}_{i,i}=0$ identically, one can set $\lambda_{i,i}=0$ without loss.
So viewing $\lambda$ as a block matrix
\begin{equation}
    \lambda = \begin{bmatrix}
        \lambda_{\star,\star} &\lambda_{\star,\cdot}\\
        \lambda_{\cdot,\star}& \hat{\lambda}
    \end{bmatrix},
\end{equation}
the first column of $\lambda$ can be taken to be zero. Hence, it suffices to consider only the remainder of the first row $\lambda_{\star,\cdot}$ and the main block $\hat\lambda$. We abbreviate this structure with a definition\footnote{Although all classical subgradient methods studied, to our knowledge, are nondegenerate, simple degenerate methods exist. For instance when $M=D=1$ and $N=2$, $W=\begin{bmatrix}
		1 & 0 \\ 2 & 3
\end{bmatrix}$ is degenerate as solving the associated PEP SDP, one has $\lambda_{1,\star}=1/4>0$.}.

\begin{definition}
    A subgradient method defined by fixed-step weights $W$ is \emph{nondegenerate} if the semidefinite programming formulation of its PEP has (i) strong duality and primal/dual attainment for the associated PEP and (ii) for each $i = 0,\dots,N$, there exists an optimal PEP solution with $f_i^W > f_\star^W$.
\end{definition}

Similar to the above complementary slackness deductions, any optimal proof must have $\gamma_\star=0$ since $\mathcal{M}_\star = M^2 - \|0\|^2 > 0$.
The following lemma uses these observations to provide a sufficient form of optimal proof certificates, independent of quantities related to the minimizer $x_\star$.

\begin{lemma} \label{lem:subgrad-structured-optimal-proof-template}
    Consider any $N\geq 1$, $M,D>0$, and any nondegenerate subgradient method defined by fixed-step weights $W$.
    Then there exists a nonnegative vector $a\in\R^{N+1}$, $\sigma\geq 0$, and a positive semidefinite matrix $\hat{Z}\in\mathbb{S}^{N+1}$ such that $\sum_{i=0}^N a_i = 1$, the following matrix is positive semidefinite
    \begin{equation}\label{eq:subgrad-structured-optimal-proof-template-psd}
        \begin{bmatrix}
            \sigma & -\frac{1}{2}a^\top\\[0.3em]
            -\frac{1}{2}a & \hat Z
        \end{bmatrix},
    \end{equation}
    and
    \begin{equation}\label{eq:subgrad-structured-optimal-proof-template}
        \mathtt{PEP}_{M}(W) - \left(f_N - \sum_{i=0}^N a_i (f_i + \langle g_i, x_0-x_i\rangle )\right)
        \geq
        \sigma D^2
        +
        \langle \hat Z,\hat G\rangle,
    \end{equation}
    where $\hat G$ denotes the principal submatrix of $G$ corresponding to subgradient-subgradient inner products. 
\end{lemma}
\begin{proof}
    Let $\lambda,\gamma,\sigma,Z$ be optimal dual PEP certificates satisfying~\eqref{eq:subgrad-primal-gap-target}. By the complementary slackness discussion above,
    we may assume $\lambda_{i,\star}=0$, $\lambda_{\star,\star} = 0$ and $\gamma_\star=0$.
    Decompose $\lambda$ and $Z$ into blocks as:
    \begin{equation*}
        \lambda = \begin{bmatrix}
        0 & a^\top\\
        0 & \hat\lambda
        \end{bmatrix}\qquad\text{and}\qquad
        Z = \begin{bmatrix}
            c & b^\top\\
            b & \hat Z
        \end{bmatrix}.
    \end{equation*}
    Then, the inner product $\langle Z,G\rangle$ can be expanded in this block form to be
    \begin{equation*}
        \langle Z,G\rangle = c\|x_0-x_\star\|^2 + 2\sum_{i=0}^N b_i\langle g_i,x_0-x_\star\rangle + \langle \hat Z,\hat G\rangle.
    \end{equation*}
    Likewise, for each $i=0,\dots,N$, we can expand
    \begin{equation*}
        \mathcal{C}_{\star,i}
        =
        f_\star - f_i - \langle g_i, x_\star-x_i\rangle
        =
        f_\star - f_i - \langle g_i, x_0-x_i\rangle + \langle g_i,x_0-x_\star\rangle.
    \end{equation*}

    Substituting these expressions into~\eqref{eq:subgrad-primal-gap-target} and recalling we have removed the first column of $\lambda$ gives
    \begin{align*}
        \mathtt{PEP}_{M}(W) - (f_N-f_\star)
        &=
        \sum_{i=0}^N a_i (f_\star - f_i - \langle g_i, x_0-x_i\rangle + \langle g_i,x_0-x_\star\rangle ) \\
        &\qquad
        + \sum_{i,j=0}^N \hat\lambda_{i,j}\mathcal{C}_{i,j}
        + \sum_{i=0}^N \gamma_i\mathcal{M}_i
        + \sigma(D^2-\|x_0-x_\star\|^2) \\
        &\qquad
        + c\|x_0-x_\star\|^2
        + 2\sum_{i=0}^N b_i\langle g_i,x_0-x_\star\rangle
        + \langle \hat Z,\hat G\rangle.
    \end{align*}
    We now compare coefficients of the free variables appearing on both sides.
    First, the coefficient of $f_\star$ on the left-hand side is $1$, while on the right-hand side it is $\sum_{i=0}^N a_i$. Hence,
    \begin{equation*}
        \sum_{i=0}^N a_i = 1.
    \end{equation*}
    Second, since the left-hand side is independent of $\langle g_i,x_0-x_\star\rangle$, the coefficient of each such term on the right-hand side must vanish, ensuring that
    \begin{equation*}
        a_i + 2b_i = 0 \qquad \forall i=0,\dots,N.
    \end{equation*}
    So we conclude that $b = -a/2$.
    Third, noting that the left-hand side is independent of $\|x_0-x_\star\|^2$, its coefficient on the right-hand side must vanish. Therefore,
    \begin{equation*}
        -\sigma + c = 0.
    \end{equation*}
    So we conclude that $c=\sigma$.
    
    Substituting these three identities back into the equality above, and canceling the $f_\star$ terms using $\sum_{i=0}^N a_i=1$, yields
    \begin{equation*}
        \mathtt{PEP}_{M}(W) - \left(f_N - \sum_{i=0}^N a_i (f_i + \langle g_i, x_0-x_i\rangle )\right)
        =
        \sum_{i,j=0}^N \hat\lambda_{i,j}\mathcal{C}_{i,j}
        +
        \sum_{i=0}^N \gamma_i\mathcal{M}_i
        +
        \sigma D^2
        +
        \langle \hat Z,\hat G\rangle.
    \end{equation*}
    Dropping the first two nonnegative summations yields~\eqref{eq:subgrad-structured-optimal-proof-template}. Finally, since $Z\succeq 0$, we have positive semidefiniteness of $\begin{bmatrix}
            \sigma & -\frac{1}{2}a^\top\\[0.3em]
            -\frac{1}{2}a & \hat Z
        \end{bmatrix}$.
\end{proof}

\begin{theorem} \label{thm:subgrad-primal-dual-certificate}
    Consider any $N\geq 1$, $M,D>0$, and any nondegenerate subgradient method defined by fixed-step weights $W$.
    Then there exists a nonnegative vector $a\in\R^{N+1}$ summing to one
     and a constant $\sigma > 0$,
    such that the method applied to any $M$-Lipschitz convex problem with $\|x_0-x_\star\|\leq D$ has
    \begin{equation}
        \underbrace{f_N - f_\star}_{\text{Primal Objective Gap}} \leq \underbrace{f_N - \alpha(x_\star)}_{\text{Primal-Dual Gap}} + \underbrace{\sigma\|z_{N+1} - x_\star\|^2}_{\text{Minimizer Estimate}}
        \leq
        \mathtt{PEP}_{M}(W)
    \end{equation}
    and
    \begin{align}
    \underbrace{f_N - f_\star}_{\text{Primal Objective Gap}} &\leq \underbrace{f_N - \min_{\|x-x_0\|\leq D} \alpha(x)}_{\text{Computable Primal-Dual Gap}}\leq \mathtt{PEP}_{M}(W)
    \end{align}
    where $\alpha(x) = \sum_{i=0}^N a_i (f(x_i) + \langle g_i, x-x_i\rangle)\in A_f\subset A_f^\star$, 
    and $z_{N+1} = x_0 - \frac{1}{2\sigma}\nabla \alpha$.
\end{theorem}
\begin{proof}
    Let $a,\sigma,\hat Z$ be as provided by Lemma~\ref{lem:subgrad-structured-optimal-proof-template}, satisfying~\eqref{eq:subgrad-structured-optimal-proof-template-psd}
    and~\eqref{eq:subgrad-structured-optimal-proof-template}.
    Since each function $x \mapsto f_i + \langle g_i, x-x_i\rangle$ is an affine minorant of $f$, and since $a_i\geq 0$ with $\sum_{i=0}^N a_i=1$, the function $\alpha$ is also an affine minorant of $f$. In particular, $\alpha\in A_f\subset A_f^\star$. Further, note the values
    \begin{equation*}
        \alpha(x_0) = \sum_{i=0}^N a_i  (f_i + \langle g_i, x_0-x_i\rangle ),
        \qquad
        \nabla \alpha = \sum_{i=0}^N a_i g_i.
    \end{equation*}
    By Lemma~\ref{lem:subgrad-structured-optimal-proof-template}, we have that
    \begin{equation*}
        \mathtt{PEP}_{M}(W) -  (f_N-\alpha(x_0) )
        \geq
        \sigma D^2 + \langle \hat Z,\hat G\rangle.
    \end{equation*}

    Next, since $\sum_{i=0}^N a_i=1$, the vector $a$ is nonzero. Therefore, since $\begin{bmatrix}
            \sigma & -\frac{1}{2}a^\top\\[0.3em]
            -\frac{1}{2}a & \hat Z
        \end{bmatrix}$ is positive semidefinite, $\sigma$ must be strictly positive. By the Schur complement, we then have $\hat Z \succeq \frac{1}{4\sigma}aa^\top.$
    Consequently,
    \begin{equation*}
        \langle \hat Z,\hat G\rangle
        \geq
        \frac{1}{4\sigma}a^\top \hat G a
        =
        \frac{1}{4\sigma}\left\|\sum_{i=0}^N a_i g_i\right\|^2
        =
        \frac{1}{4\sigma}\|\nabla \alpha\|^2.
    \end{equation*}

    From the last two displays, we conclude that
    \begin{equation} \label{eq:combined-display}
        \mathtt{PEP}_{M}(W) - (f_N-\alpha(x_0))
        \geq
        \sigma D^2 + \frac{1}{4\sigma}\|\nabla \alpha\|^2.
    \end{equation}
    We will bound \eqref{eq:combined-display} in two different ways.

    First, we can combine \eqref{eq:combined-display} with the following two inequalities:
    \begin{align*}
        f_\star &\geq \alpha(x_\star) = \alpha(x_0) + \langle{x_\star - x_0, \nabla \alpha} \rangle\\
\|{z_{N+1}-x_\star}\|^2 &\leq D^2 + \frac{1}{4\sigma^2}\|\nabla\alpha\|^2 + \frac{1}{\sigma}\langle{\nabla\alpha, x_\star - x_0}\rangle
    \end{align*}
    to deduce that
    \begin{align*}
        f_N - f_\star &\leq f_N - \alpha(x_\star) + \sigma\|z_{N+1}-x_\star\|^2\\
        &\leq f_N - \alpha(x_0) + \sigma D^2 + \frac{1}{4\sigma}\|\nabla\alpha\|^2\\
        &\leq \mathtt{PEP}_{M}(W).
    \end{align*}
    
    Alternatively, by the AM-GM inequality,
    \begin{align*}
        \mathtt{PEP}_{M}(W) - (f_N-\alpha(x_0))
        \geq
        \sigma D^2 + \frac{1}{4\sigma}\|\nabla \alpha\|^2 \geq D\|\nabla\alpha\|.
    \end{align*}
    Rearranging this inequality gives the second claim.
\end{proof}
The first bound in Theorem~\ref{thm:subgrad-primal-dual-certificate} shows that the primal-dual gap $f(x_N)-\alpha(x_\star)$ converges not only as fast as the worst-case primal gap, but outperforms it by $\sigma\|z_{N+1} - x_\star\|^2$. Hence, in the tight worst-case where the primal gap (and hence primal-dual gap) equals $\mathtt{PEP}_{M}(W)$, one must have that the minimizer exactly occurs at $z_{N+1} = x_0 - \frac{1}{2\sigma}\nabla \alpha$. So the dual solution $\alpha$ not only provides a dual lower bound but also indicates where the minimizer must be in tight, worst-case instances. As the true minimizer $x_\star$ differs from the worst-case placement $z_{N+1}$, the method's guarantee improves at a quadratic rate.


Finally, we note two subtleties in Theorem~\ref{thm:subgrad-primal-dual-certificate}. First, while the tightness of the PEP reformulation required $d\geq N+2$, the above upper bounding guarantee holds for problems of any dimension. Second, the necessary dual certificate combines subgradient lower bounds from all of the iterates $x_0,\dots,x_N$. As a result, numerically constructing this certificate requires one to compute a subgradient at the terminal iterate $x_N$ (if $a_N>0$), which may otherwise not be needed.

\subsection{Example Subgradient Methods and their Implicit Dual Certificates}

Below, we consider three subgradient methods from the literature, each known to admit a PEP certificate proving an exactly optimal worst-case rate of $MD/\sqrt{N+1}$. Since each of these methods is exactly optimal, its worst-case is attained by the maximin hard instance presented in~\cite[Appendix A]{drori2016Kelley}. Their known proof and this hard instance establish strong duality and attainment. In particular, on this instance, no method can reach the minimum objective value in $i < N+1$ iterations. Hence, $f_i^W>f^W_\star$. So each optimal method below is nondegenerate.

For each considered method, the affine function certifying this optimal objective gap rate is immediately available by examining the existing convergence proofs of these methods in the literature (in particular, the value of $\lambda_{\star,i}$ which sets $a_i$). For each method, the considered proof has $\lambda_{\star,i}$ constant, leading the dual model in each case to be the uniform averaging of the subgradient affine models seen. Indeed, this is necessary for all minimax optimal subgradient methods and optimal proofs, as shown by the complete characterization of~\cite{zoll2026complete}.

{\bf The Averaged Subgradient Method.} The averaged subgradient method iterates
\begin{align*}
    x_{n} &= x_{n-1} - \frac{D}{M\sqrt{N+1}} g_{n-1} \qquad \forall n=1,\dots,N-1\\
    x_{N} &= \frac{1}{N+1}\left[\sum_{i=0}^{N-1} x_i + \left(x_{N-1} - \frac{D}{M\sqrt{N+1}} g_{N-1}\right)\right].
\end{align*}
This amounts to the subgradient method with a fixed stepsize $\frac{D}{M\sqrt{N+1}}$ for $N$ iterations, where the terminal iterate is set as the average of these iterations. The presentation above combines the final step and averaging step to fit the standard fixed-step model~\eqref{eq:general-form-subgrad-methods}.

A known optimal proof for this method, establishing that it has $f(x_N)-f(x_\star)\leq MD/\sqrt{N+1}$, has constant coefficients $\lambda_{\star,i}$ equal to $1/(N+1)$ and $\sigma = M/(2D\sqrt{N+1})$. Hence $a_i=1/(N+1)$ and so the averaged subgradient method's standard primal objective gap convergence can be strengthened to the guarantees that
\begin{align} \label{eq:subgrad-strengthen-rate1}
    f(x_N) - \alpha(x_\star) + \sigma\left\|z_{N+1} - x_\star\right\|^2 &\leq \frac{MD}{\sqrt{N+1}}\\
    f(x_N) - \alpha(x_0) + D\|\nabla \alpha\| &\leq \frac{MD}{\sqrt{N+1}} \label{eq:subgrad-strengthen-rate2}
\end{align}
where the dual certificate is given by the affine minorant $\alpha(x) = \frac{1}{N+1} \sum_{i=0}^{N}\left(f(x_i) + \langle g_i, x-x_i\rangle \right)$ and the estimate of the minimizer is given by $z_{N+1} = x_0 - \frac{D}{M\sqrt{N+1}}\sum_{i=0}^N g_i$.

{\bf The Final Step Subgradient Method.} Recently, Zamani and Glineur~\cite{ZamaniGlineur2025} proposed a subgradient method with nonconstant stepsizes whose terminal iterate (without averaging) possesses the same optimal convergence guarantee as seen above. Formally, they consider the iteration for $n=1,\dots,N$
\begin{equation*}
    x_{n} = x_{n-1} - \frac{D(N+1-n)}{M(N+1)^{3/2}} g_{n-1}.
\end{equation*}
Their proof also uses uniform multipliers $\lambda_{\star,i} = 1/(N+1)$ and $\sigma = M/(2D\sqrt{N+1})$. So their method implicitly constructs the same affine minorant $\alpha(x) = \frac{1}{N+1} \sum_{i=0}^{N}\left(f(x_i) + \langle g_i, x-x_i\rangle \right)$ and constructs the same minimizer estimate $z_{N+1} = x_0 - \frac{D}{M\sqrt{N+1}}\sum_{i=0}^N g_i$. These strengthen Theorem 5.1 of~\cite{ZamaniGlineur2025} to ensure that their final iterate has the same stronger primal-dual gap bounds~\eqref{eq:subgrad-strengthen-rate1} and~\eqref{eq:subgrad-strengthen-rate2}.

{\bf The SSEP Method.} Another optimal fixed-step method of a momentum type was designed by Drori and Taylor~\cite{drori2019efficient}. They considered the following subgradient SSEP method for $n=1,\dots,N$
\begin{align*}
    y_n &= \frac{n}{n+1}x_{n-1} + \frac{1}{n+1}x_0,\\
    d_n &= \frac{1}{n+1}\sum_{i=0}^{n-1} g_i,\\
    x_n &= y_n - \frac{D}{M\sqrt{N+1}}\, d_n
\end{align*}
which can be recursively expanded to verify that it takes the fixed-step form~\eqref{eq:general-form-subgrad-methods}. Corollary 3 of~\cite{drori2019efficient} establishes that this has an optimal $MD/\sqrt{N+1}$ convergence rate. Again, their proof uses multipliers $\lambda_{\star,i}=1/(N+1)$ and $\sigma = M/(2D\sqrt{N+1})$. So the implicit dual constructions take the same form. Hence, their Corollary 3 can be strengthened to ensure~\eqref{eq:subgrad-strengthen-rate1} and~\eqref{eq:subgrad-strengthen-rate2}.

Our implicit dual theory provides an alternative understanding of the SSEP momentum method. Consider the affine minorant corresponding to the partial dual certificate built prior to iteration $n$
$$\alpha_n(x) = \frac{1}{n}\sum_{i=0}^{n-1} (f(x_i) + \langle g_i, x - x_i\rangle). $$
Noting that this dual solution has $d_n = \frac{n}{n+1}\nabla \alpha_n$, the SSEP method can be written in primal-dual terms, alternating between averaging and descent on the dual minorant. In these terms, the SSEP method is initialized with $\alpha_1(x) = f(x_0) + \langle g_0, x - x_0\rangle$ and iterates for $n=1,\dots,N$
\begin{align*}
    x_n &=  \frac{n}{n+1}\left(x_{n-1} - \frac{D}{M\sqrt{N+1}} \nabla \alpha_n\right)+ \frac{1}{n+1}x_0 \\
    \alpha_{n+1} &= \frac{n}{n+1}\alpha_{n} + \frac{1}{n+1} (f(x_{n}) + \langle g_{n}, \cdot - x_{n}\rangle).
\end{align*}

\section{Dual Constructions for Fixed-Step Gradient Methods} \label{sec:grad-development}
The same progression used for subgradient methods also applies to gradient methods. Consider any fixed-step gradient method~\eqref{eq:general-form-subgrad-methods} with weights $W$ for $L$-smooth convex minimization. Then its worst-case performance (invoking Theorem~\ref{thm:SmoothConvexInterpolation} as an interpolation result) is given by
\begin{align}
    \mathtt{PEP}_{L}(W) &:= \begin{cases}
        \max_{f,x_i} &f(x_N)-f(x_\star)\\
        \mathrm{s.t.} & x_n = x_0 - \sum_{i=0}^{n-1}W_{n-1,i}g_i \quad \forall n=1,\dots,N\\
        & g_i = \nabla f(x_i) \qquad \forall i\in\cI\\
        & f\text{ is convex, $L$-smooth, and minimized at $x_\star$}\\
        & \|x_0-x_\star\|\leq D
    \end{cases} \nonumber\\
    &= \begin{cases}
        \max_{(x_i,f_i,g_i)_{i\in\cI}} & f_N - f_\star\\
        \mathrm{s.t.} & x_n = x_0 - \sum_{i=0}^{n-1}W_{n-1,i}g_i \quad \forall n=1,\dots,N\\
        & g_\star = 0\\
        & \mathcal{Q}_{i,j} \geq 0 \qquad \forall i,j\in\cI\\
        & \mathcal{D} \geq 0.
    \end{cases} \label{eq:gradient-pep}
\end{align}
As before, this problem has a well-known reformulation as a semidefinite program. This follows by substituting each $x_n$ and $g_\star$ for their equality definitions. This makes the above problem linear in the function values $f=(f_\star,f_0,\dots,f_N)$ and linear in the quadratic form $G = P^\top P$, with $P = \begin{bmatrix} x_0-x_\star, g_0,\dots, g_N \end{bmatrix}$ which must be positive semidefinite. Provided $d\geq N+2$, the resulting semidefinite program is equivalent.

The dual semidefinite program, with nonnegative multipliers $\lambda_{i,j}$ for the $\mathcal{Q}_{i,j}\geq 0$ constraints, takes the same general form. A bound $f_N-f_\star \leq r$ is proven by nonnegative multipliers $\lambda_{i,j},\sigma$ and a positive semidefinite $Z$ such that the following identity holds
\begin{equation*}
    r - (f_N-f_\star) = \sum_{i,j\in\cI} \lambda_{i,j}\mathcal{Q}_{i,j} + \sigma \mathcal{D} + \langle Z, G\rangle,
\end{equation*}
fixing $x_n = x_0 - \sum_{i=0}^{n-1}W_{n-1,i}g_i$ and $g_\star=0$, and treating $f_i$ and each inner product from $G$ as free variables. The dual optimization problem finds the best proof of this form (i.e., minimizing $r$). Given that strong duality holds, one can find multipliers such that
\begin{equation} \label{eq:smooth-general-proof-template}
    \mathtt{PEP}_{L}(W) - (f_N-f_\star) = \sum_{i,j\in\cI} \lambda_{i,j}\mathcal{Q}_{i,j} + \sigma \mathcal{D} + \langle Z, G\rangle.
\end{equation}

\subsection{Primal-Dual Guarantees for Smooth Convex Minimization}

We restrict our attention to {\it nondegenerate} gradient methods\footnote{Again, simple degenerate methods exist. For instance when $L=D=1$ and $N=2$, $W=\begin{bmatrix}
		1 & 0 \\ 2 & 3
\end{bmatrix}$ is degenerate as solving the associated PEP SDP, one has $\lambda_{1,\star}=7/2>0$.}, defined as those whose semidefinite programming PEP formulation has (i) strong duality and primal/dual attainment for the associated PEP and (ii) for each $i=0,\dots,N$, a tight instance exists with realized values $f_i^W - \frac{1}{2L}\|g^W_i\|^2 > f_\star^W$.

Note that if $f_i^W - \frac{1}{2L}\|g^W_i\|^2 = f_\star^W$, then the classic descent lemma (i.e., $\mathcal{Q}_{i,\star}\geq 0$, noting $g_\star=0$) ensures that the algorithm is one gradient descent step (with stepsize $1/L$) away from reaching the exact minimum value.

We can apply the same reasoning used to prove Lemma~\ref{lem:subgrad-structured-optimal-proof-template} to conclude the following lemma. As before, the constructed weights $a_i$ below are exactly the optimal proof's $\lambda_{\star,i}$ multipliers. In the case of smooth optimization, we further deduce the exact value of $\sigma$ in terms of $D$ and $\mathtt{PEP}_{L}(W)$.

\begin{lemma} \label{lem:grad-structured-optimal-proof-template}
    Consider any $N\geq 1$, $L,D>0$, and any nondegenerate gradient method defined by fixed-step weights $W$.
    Then there exists a nonnegative vector $a\in\R^{N+1}$, nonnegative multipliers $\hat\lambda_{i,j}$ for $i,j = 0,\dots, N$, and a positive semidefinite matrix $\hat{Z}\in\mathbb{S}^{N+1}$ such that $\sum_{i=0}^N a_i = 1$, the following matrix is positive semidefinite
    \begin{equation}\label{eq:grad-structured-optimal-proof-template-psd}
        \begin{bmatrix}
            \sigma & -\frac{1}{2}a^\top\\[0.3em]
            -\frac{1}{2}a & \hat Z
        \end{bmatrix},
    \end{equation}
    and
    \begin{equation}\label{eq:grad-structured-optimal-proof-template}
        \mathtt{PEP}_{L}(W) - \left(f_N - \sum_{i=0}^N a_i \left(f_i + \langle g_i, x_0-x_i\rangle + \frac{1}{2L}\|g_i\|^2\right)\right)
        =
        \sum_{i,j=0}^N \hat\lambda_{i,j}\mathcal{Q}_{i,j}
        +
        \sigma D^2
        +
        \langle \hat Z,\hat G\rangle,
    \end{equation}
    where $\sigma=\frac{\mathtt{PEP}_{L}(W)}{D^2}>0$ and $\hat G$ denotes the principal submatrix of $G$ corresponding to gradient-gradient inner products.
\end{lemma}
\begin{proof}
    The proof of this result follows almost identically to the previous lemma, omitting terms related to Lipschitzness $\gamma_i\mathcal{M}_i$ and replacing each $\mathcal{C}_{i,j}$ by $\mathcal{Q}_{i,j}$. 
    Since the $\mathcal{M}_i$ terms are omitted in~\eqref{eq:smooth-general-proof-template}, one can compare constant terms on the left- and right-hand sides. The only scalar on the right-hand side is $\sigma D^2$. From this, we deduce that any dual feasible solution must have $\sigma=\frac{\mathtt{PEP}_{L}(W)}{D^2}$.
    Since we replaced $\mathcal{C}_{i,j}$ by $\mathcal{Q}_{i,j}$, observe that whereas
    \begin{equation*}
        \mathcal{C}_{\star,i}
        =
        f_\star - f_i - \langle g_i, x_\star-x_i\rangle
        =
        f_\star - f_i - \langle g_i, x_0-x_i\rangle + \langle g_i,x_0-x_\star\rangle,
    \end{equation*}
    we now have an additional $\frac{1}{2L}\|g_i\|^2$ term as
    \begin{equation*}
        \mathcal{Q}_{\star,i}
        =
        f_\star - f_i - \langle g_i, x_\star-x_i\rangle - \frac{1}{2L}\|g_i\|^2
        =
        f_\star - f_i - \langle g_i, x_0-x_i\rangle + \langle g_i,x_0-x_\star\rangle - \frac{1}{2L}\|g_i\|^2.
    \end{equation*}
    Carrying this term forward, the remainder of the proof goes through without modification.
\end{proof}
This lemma directly supports a similar main theorem for smooth convex optimization. The proof of this result follows the same sequence of algebraic steps. The individual affine terms in the previous lemma
$$ \left(f_i + \langle g_i, x-x_i\rangle + \frac{1}{2L}\|g_i\|^2\right) $$
are no longer guaranteed to be minorants due to the added gradient norm squared term. So they may not lie in $A_f$. However, cocoercivity~\eqref{eq:cocoercive} (i.e., $\mathcal{Q}_{\star,i}\geq 0$) ensures that they are ``minorants at the minimizer'', lying in $A_f^\star$. As discussed in Section~\ref{sec:primal-dual}, this difference has no effect on the validity of the resulting primal-dual gap. Noting that $A_f^\star$ is also closed under convex combinations, the same proof as the previous theorem gives the following extension.

\begin{theorem} \label{thm:grad-primal-dual-certificate}
    Consider any $N\geq 1$, $L,D>0$, and any nondegenerate gradient method defined by fixed-step weights $W$.
    Then there exists a nonnegative vector $a\in\R^{N+1}$, summing to one, such that the method applied to any $L$-smooth convex problem with $\|x_0-x_\star\|\leq D$ has
    \begin{equation}
        \underbrace{f_N - f_\star}_{\text{Primal Objective Gap}} \leq \underbrace{f_N - \alpha(x_\star)}_{\text{Primal-Dual Gap}} + \underbrace{\frac{\mathtt{PEP}_{L}(W)}{D^2}\|z_{N+1} - x_\star\|^2}_{\text{Minimizer Estimate}}
        \leq
        \mathtt{PEP}_{L}(W)
    \end{equation}
    and
    \begin{equation}
        \underbrace{f_N - f_\star}_{\text{Primal Objective Gap}} \leq \underbrace{f_N -  (\alpha(x_0) - D\|\nabla \alpha\| )}_{\text{Computable Primal-Dual Gap}} \leq \mathtt{PEP}_{L}(W)
    \end{equation}
    where $\alpha(x) = \sum_{i=0}^N a_i\left(f(x_i) + \langle g_i, x-x_i\rangle + \frac{1}{2L}\|g_i\|^2\right)$, which lies in $A_f^\star$, and $z_{N+1} = x_0 - \frac{D^2}{2\mathtt{PEP}_{L}(W)}\nabla \alpha $.
\end{theorem}
Again, note that computing the affine dual certificate for a given algorithm requires computing a gradient at the terminal iterate (provided $a_N>0$). So in order to report the final objective value $f(x_N)$ and a dual lower bound $(\alpha(x_0) - D\|\nabla \alpha\|)$, one additional round of first-order queries is needed.

\subsection{Example Gradient Methods and their Implicit Dual Certificates}
Here, we consider three methods, gradient descent, the Optimized Gradient Method, and Nesterov's Fast Gradient Method. The first two have well-known PEP-style convergence proofs and attain their worst-case convergence rate on a Huber-type function. Together, these establish strong duality and attainment. On each method's tight Huber function instance, the iterates are all more than one gradient descent step away from the minimizer. As a result, the first two methods are nondegenerate.

The exactly optimal PEP certificate for Nesterov's Fast Gradient Method, to our knowledge, has not been analytically determined. A feasible, nearly optimal PEP certificate was identified by~\cite{kim2016OGM} with each $\lambda_{i,\star}=0$. So our theory can still be applied to extract primal-dual guarantees from this known, not-quite-tight certificate.

{\bf Gradient Descent.} First, consider gradient descent, which iterates for $n=1,\dots,N$
$$
x_n = x_{n-1} - \frac{1}{L}g_{n-1}.
$$
Drori and Teboulle~\cite{drori2014performance} gave a tight convergence guarantee for gradient descent, ensuring $f_N-f_\star\leq\mathtt{PEP}_{L}(W) = LD^2/(4N+2)$. Their dual proof sets multipliers on the star-row inequalities as
\begin{equation*}
\lambda_{\star,0}=\frac{1}{2N},\qquad
\lambda_{\star,i}
=\frac{2N+1}{(2N-i)(2N+1-i)}
\quad (i=1,\dots,N-1),
\qquad
\lambda_{\star,N}=\frac{1}{N+1}.
\end{equation*}
Hence, these are exactly the coefficients $a_i$ in the implicitly constructed dual certificate, and they satisfy $\sum_{i=0}^N a_i=1$. Therefore, gradient descent's known tight primal objective gap guarantee immediately strengthens to our pair of primal-dual guarantees with the dual certificate as the affine function
$$
\alpha(x)=\sum_{i=0}^{N} a_i\left(f(x_i)+\langle g_i,x-x_i\rangle+\frac{1}{2L}\|g_i\|^2\right),
$$
and coefficients $a_i = \lambda_{\star, i}$. Note that these are nonnegative and sum to one as needed.

{\bf The Optimized Gradient Method (OGM).} Second, we consider OGM as defined by~\cite{kim2016OGM}, which is known to be exactly minimax optimal for $L$-smooth convex minimization (provided $d\geq N+2$). This method first computes $g_0=\nabla f(x_0)$, sets $\tau_0=2$ and $z_1=x_0-\frac{2}{L}g_0$, and then iterates for $n=1,\dots,N$
\begin{align}
    \tau_n &= \begin{cases}
            \tau_{n-1} + 1 + \sqrt{1+2\tau_{n-1}} \qquad & \text{if $n<N$}\\
            \tau_{n-1} + \frac{1 + \sqrt{1+ 4\tau_{n-1}}}{2} \qquad & \text{if $n=N$,}
        \end{cases} \label{eq:OGM_tau_def}\\
    x_n &= \frac{\tau_{n-1}}{\tau_n} \left(x_{n-1} - \frac{1}{L} g_{n-1}\right) +  \frac{\tau_n - \tau_{n-1}}{\tau_n}z_{n},\\
    z_{n+1} &= z_{n} - \frac{\tau_n-\tau_{n-1}}{L}g_n.
\end{align}
Above, we use the $\tau_n$ scalar sequence of~\cite{SPGM} to define OGM. This is mathematically equivalent to the recursive definition and inductive analyses using $\theta^2_n$ sequences in the literature~\cite{TaylorBach2019a,park2023factor}.
One known convergence proof of OGM's optimal rate sets $\lambda_{\star,0}=\frac{2}{\tau_N}$ and each $\lambda_{\star,n} = \frac{\tau_{n}-\tau_{n-1}}{\tau_N}$ for $n=1,\dots,N$, proving $\mathtt{PEP}_{L}(W) = \frac{LD^2}{2\tau_N}$. So the corresponding implicit dual certificate constructed by OGM is the affine function
\begin{equation*}
    \alpha(x)
    =
    \frac{2}{\tau_N}\left(f(x_0) + \langle g_0, x-x_0\rangle +\frac{1}{2L}\|g_0\|^2\right)
    +
    \sum_{i=1}^{N}\frac{\tau_i-\tau_{i-1}}{\tau_N}\left(f(x_i) + \langle g_i, x-x_i\rangle +\frac{1}{2L}\|g_i\|^2\right).
\end{equation*}
The slope of this affine function is then exactly related to $(z_{N+1}-x_0)$ by
\begin{equation*}
    \nabla \alpha
    =
    \frac{2}{\tau_N}g_0 + \sum_{i=1}^N \frac{\tau_i-\tau_{i-1}}{\tau_N}g_i
    =
    -\frac{L}{\tau_N}(z_{N+1}-x_0).
\end{equation*}
Hence, the constructed $z_{N+1}$ from OGM exactly agrees with our theorem's $z_{N+1}$. So the auxiliary sequence precisely provides the location that the minimizer $x_\star$ must be in tight instances where the worst-case primal gap (and hence primal-dual gap) occurs. Thus, at each step, OGM estimates the worst-case minimizer's location and averages toward it while taking a gradient descent step.

Further, one can view the auxiliary/momentum sequence $z_{n+1}$ as tracking the slope of the aggregate dual solution implicitly constructed by the algorithm's optimal proof. As we observed with the SSEP method, OGM can then be stated in primal-dual terms as alternating averaging with descent of the objective and descent on the dual certificate. With the $\tau_n$ sequence defined as above, OGM initializes $\alpha_1(x) = f(x_0) + \langle g_0, x - x_0\rangle + \frac{1}{2L}\|g_0\|^2$ (which lies in $A_f^\star$) and iterates
\begin{align*}
    x_n &= \frac{\tau_{n-1}}{\tau_n} \left(x_{n-1} - \frac{1}{L} g_{n-1}\right) +  \frac{\tau_n - \tau_{n-1}}{\tau_n}\left(x_0  - \frac{\tau_{n-1}}{L}\nabla\alpha_n\right)\\
    \alpha_{n+1} &= \frac{\tau_{n-1}}{\tau_n} \alpha_n +  \frac{\tau_n - \tau_{n-1}}{\tau_n}\left(f(x_n) + \langle g_n, \cdot-x_n\rangle + \frac{1}{2L}\|g_n\|^2\right),
\end{align*}
ultimately reporting $x_N$ and $\alpha_{N+1}$. Taking either of these perspectives, our Theorem~\ref{thm:grad-primal-dual-certificate} provides OGM with two strengthened statements of its optimal convergence rate: (i) OGM has primal-dual gap $f(x_N)-\alpha_{N+1}(x_\star)$ bounded by the same optimal rate and strictly improves when the minimizer differs from its auxiliary $z_{N+1}$, (ii) OGM has a computable primal-dual gap converging at the same optimal rate as it alternates between building $x_n$ and building $\alpha_{n+1}$.

{\bf Nesterov's Fast Gradient Method.} Finally, we consider the fast gradient method, iterating 
\begin{align*}
	t_n &= \frac{1+\sqrt{1+4t_{n-1}^2}}{2},\\
	x_n &= \frac{t_n - 1}{t_n} \left(x_{n-1} - \frac{1}{L} g_{n-1}\right) +  \frac{1}{t_n}z_{n},\\
	z_{n+1} &= z_{n} - \frac{t_n}{L}g_n
\end{align*}
with $t_0=1$ and $z_1=x_0 - t_0 g_0/L$. A PEP-style convergence guarantee for this method was given by~\cite{kim2016OGM} by considering $\sigma = \frac{L}{2t_N^2}$, a sparse $\lambda$ selection having only nonzero values for $\lambda_{i,i+1} = \frac{t_i^2}{t_N^2}$ and $\lambda_{\star,i} = \frac{t_i}{t_N^2}$, and finally,
\begin{align*}
	Z &= \frac{1}{2t_N^2}
	\left[
	\begin{pmatrix}
		\sqrt{L}\\[0.2em]
		-\frac{t_0}{\sqrt{L}}\\
		\vdots\\
		-\frac{t_N}{\sqrt{L}}
	\end{pmatrix}
	\begin{pmatrix}
		\sqrt{L} &
		-\frac{t_0}{\sqrt{L}} &
		\cdots &
		-\frac{t_N}{\sqrt{L}}
	\end{pmatrix}
	+
	\frac{1}{L}\operatorname{diag}(0,t_0^2,\dots,t_{N-1}^2,0)
	\right].
\end{align*}
(Note this proof matches Nesterov's classic rate, but is not exactly tight in the PEP-sense.) This certificate proves a guarantee of
$$
f(x_N)-f(x_\star)\leq \frac{L\|x_0-x_\star\|^2}{2t_N^2}
\leq \frac{2L\|x_0-x_\star\|^2}{(N+2)^2}.
$$
As we did with OGM, applying our reformulation of guarantees, Nesterov's Fast Gradient Method can be seen as implicitly building the dual model
\[
\alpha_n(x)
=
\sum_{i=0}^{n-1}
\frac{t_i}{t_{n-1}^2}
\left(
f(x_i)+\langle g_i,x-x_i\rangle+\frac{1}{2L}\|g_i\|^2
\right)
\]
at each iteration. The slope of this affine function is given by $\nabla \alpha_n
=
\sum_{i=0}^{n-1}\frac{t_i}{t_{n-1}^2}g_i
=
-\frac{L}{t_{n-1}^2}(z_n-x_0)$.
Then, one can write the update in primal-dual form. In these terms, Nesterov's method initializes $\alpha_1(x)=f(x_0)+\langle g_0,x-x_0\rangle+\frac{1}{2L}\|g_0\|^2,$ and then iterates
\begin{align*}
	x_n
	&=
	\frac{t_n-1}{t_n}
	\left(x_{n-1}-\frac{1}{L}g_{n-1}\right)
	+
	\frac{1}{t_n}
	\left(x_0-\frac{t_{n-1}^2}{L}\nabla\alpha_n\right),\\
	\alpha_{n+1}
	&=
	\frac{t_n-1}{t_n}\alpha_n
	+
	\frac{1}{t_n}
	\left(
	f(x_n)+\langle g_n,\cdot-x_n\rangle+\frac{1}{2L}\|g_n\|^2
	\right).
\end{align*}
While this runs, the same classic $O(1/N^2)$ primal guarantee applies more strongly, ensuring that the primal-dual gap with estimate error ($f(x_N)-\alpha_{N+1}(x_\star) + \frac{L}{2t_N^2}\|z_{N+1}-x_\star\|^2$) and the computable primal-dual gap ($f(x_N)-\left(\alpha_{N+1}(x_0)-D\|\nabla\alpha_{N+1}\|\right)$) are both at most $\frac{LD^2}{2t_N^2}$.

\section{Conclusion}
Dual interpretations of primal methods have occurred throughout the optimization literature. In nonsmooth optimization, Nesterov~\cite{Nesterov2009PD} connected traditional primal subgradient methods with a dual perspective, enabling new averaging methods. More recently, Grimmer and Li~\cite{GrimmerLi2025} explicitly connected dual averaging to the kind of primal-dual minimax perspectives considered here.
In smooth optimization, several previously discussed works~\cite{Diakonikolas2019,AllenZhu2017,LanZhou2018,Drusvyatskiy2018,gutman2022,WangAbernethyLevy2024,burns2026} presented dual analyses and explanations of momentum. Here, we provide a general mechanism from which dual phenomena arise. Equivalent primal-dual convergence theory must exist as a necessary consequence of structured PEP proofs and nondegeneracy.

For Lipschitz convex minimization and smooth convex minimization, we have shown that attempts to design nondegenerate algorithms and proofs for reducing the objective gap inherently prove stronger guarantees on primal-dual gaps. Future efforts to design optimized algorithms may benefit from pursuing primal-dual guarantees directly. Building on our understanding of auxiliary sequences in momentum methods as an implicit dual proof construction may also lead to new methods.

{\small \paragraph{Acknowledgments.} This work was supported in part by the Air Force Office of Scientific Research. Benjamin Grimmer was supported as a fellow of the Alfred P. Sloan Foundation.}

{\small
\bibliographystyle{unsrt}
\bibliography{bibliography}
}

\end{document}